\documentclass[12pt,reqno]{amsart}
\usepackage[usenames,dvipsnames]{xcolor}
\usepackage{etex}
\usepackage{amscd,amssymb,amsmath,multicol,tikz,float}
\usepackage{hyperref}
\usepackage{cleveref}

\RequirePackage{etex}
\hypersetup{colorlinks={true},linkcolor={OliveGreen},citecolor=Maroon}
\begin{document}

\restylefloat{table}
\newtheorem{thm}[equation]{Theorem}
\numberwithin{equation}{section}
\newtheorem{cor}[equation]{Corollary}
\newtheorem{expl}[equation]{Example}
\newtheorem{rmk}[equation]{Remark}
\newtheorem{conv}[equation]{Convention}
\newtheorem{claim}[equation]{Claim}
\newtheorem{lem}[equation]{Lemma}
\newtheorem{sublem}[equation]{Sublemma}
\newtheorem{conj}[equation]{Conjecture}
\newtheorem{defin}[equation]{Definition}
\newtheorem{diag}[equation]{Diagram}
\newtheorem{prop}[equation]{Proposition}
\newtheorem{notation}[equation]{Notation}
\newtheorem{tab}[equation]{Table}
\newtheorem{fig}[equation]{Figure}
\newcounter{bean}
\renewcommand{\theequation}{\thesection.\arabic{equation}}
\numberwithin{figure}{section}

\raggedbottom \voffset=-.7truein \hoffset=0truein \vsize=8truein
\hsize=6truein \textheight=8truein \textwidth=6truein
\baselineskip=18truept

\def\mapright#1{\ \smash{\mathop{\longrightarrow}\limits^{#1}}\ }
\def\mapleft#1{\smash{\mathop{\longleftarrow}\limits^{#1}}}
\def\mapup#1{\Big\uparrow\rlap{$\vcenter {\hbox {$#1$}}$}}
\def\mapdown#1{\Big\downarrow\rlap{$\vcenter {\hbox {$\ssize{#1}$}}$}}
\def\mapne#1{\nearrow\rlap{$\vcenter {\hbox {$#1$}}$}}
\def\mapse#1{\searrow\rlap{$\vcenter {\hbox {$\ssize{#1}$}}$}}
\def\mapr#1{\smash{\mathop{\rightarrow}\limits^{#1}}}
\def\ss{\smallskip}
\def\s{\sigma}
\def\l{\lambda}
\def\vp{v_1^{-1}\pi}
\def\at{{\widetilde\alpha}}

\def\sm{\wedge}
\def\la{\langle}
\def\ra{\rangle}
\def\ev{\text{ev}}
\def\od{\text{od}}
\def\on{\operatorname}
\def\ol#1{\overline{#1}{}}
\def\spin{\on{Spin}}
\def\cat{\on{cat}}
\def\Lbar{\overline{\Lambda}}
\def\qed{\quad\rule{8pt}{8pt}\bigskip}
\def\ssize{\scriptstyle}
\def\a{\alpha}
\def\bz{{\Bbb Z}}
\def\Rhat{\hat{R}}
\def\im{\on{im}}
\def\ct{\widetilde{C}}
\def\ext{\on{Ext}}
\def\sq{\on{Sq}}
\def\eps{\epsilon}
\def\ar#1{\stackrel {#1}{\rightarrow}}
\def\br{{\mathbf R}}
\def\bC{{\mathbf C}}
\def\bA{{\mathbf A}}
\def\bB{{\mathbf B}}
\def\bD{{\mathbf D}}
\def\bC{{\mathbf C}}
\def\bh{{\mathbf H}}
\def\bQ{{\mathbf Q}}
\def\bP{{\mathbf P}}
\def\bx{{\mathbf x}}
\def\bo{{\mathbf{bo}}}
\def\dh{\widehat{d}}
\def\si{\sigma}
\def\Vbar{{\overline V}}
\def\dbar{{\overline d}}
\def\wbar{{\overline w}}
\def\Sum{\sum}
\def\tfrac{\textstyle\frac}

\def\tb{\textstyle\binom}
\def\Si{\Sigma}
\def\w{\wedge}
\def\equ{\begin{equation}}
\def\b{\beta}
\def\G{\Gamma}
\def\L{\Lambda}
\def\g{\gamma}
\def\d{\delta}
\def\k{\kappa}
\def\psit{\widetilde{\Psi}}
\def\tht{\widetilde{\Theta}}
\def\psiu{{\underline{\Psi}}}
\def\thu{{\underline{\Theta}}}
\def\aee{A_{\text{ee}}}
\def\aeo{A_{\text{eo}}}
\def\aoo{A_{\text{oo}}}
\def\aoe{A_{\text{oe}}}
\def\vbar{{\overline v}}
\def\endeq{\end{equation}}
\def\sn{S^{2n+1}}
\def\zp{\mathbf Z_p}
\def\cR{{\mathcal R}}
\def\P{{\mathcal P}}
\def\cQ{{\mathcal Q}}
\def\cj{{\cal J}}
\def\zt{{\mathbf Z}_2}
\def\bs{{\mathbf s}}
\def\bof{{\mathbf f}}
\def\bq{{\mathbf Q}}
\def\be{{\mathbf e}}
\def\Hom{\on{Hom}}
\def\ker{\on{ker}}
\def\kot{\widetilde{KO}}
\def\coker{\on{coker}}
\def\da{\downarrow}
\def\colim{\operatornamewithlimits{colim}}
\def\zphat{\bz_2^\wedge}
\def\io{\iota}
\def\om{\omega}
\def\Prod{\prod}
\def\e{{\cal E}}
\def\zlt{\Z_{(2)}}
\def\exp{\on{exp}}
\def\abar{{\overline a}}
\def\xbar{{\overline x}}
\def\ybar{{\overline y}}
\def\zbar{{\overline z}}
\def\mbar{{\overline m}}
\def\nbar{{\overline n}}
\def\sbar{{\overline s}}
\def\kbar{{\overline k}}
\def\bbar{{\overline b}}
\def\et{{\widetilde E}}
\def\ni{\noindent}
\def\tsum{\textstyle \sum}
\def\coef{\on{coef}}
\def\den{\on{den}}
\def\lcm{\on{l.c.m.}}
\def\Ext{\operatorname{Ext}}
\def\iso{\approx}
\def\lra{\longrightarrow}
\def\vi{v_1^{-1}}
\def\ot{\otimes}
\def\psibar{{\overline\psi}}
\def\thbar{{\overline\theta}}
\def\mhat{{\hat m}}
\def\exc{\on{exc}}
\def\ms{\medskip}
\def\ehat{{\hat e}}
\def\etao{{\eta_{\text{od}}}}
\def\etae{{\eta_{\text{ev}}}}
\def\dirlim{\operatornamewithlimits{dirlim}}
\def\gt{\widetilde{L}}
\def\lt{\widetilde{\lambda}}
\def\st{\widetilde{s}}
\def\ft{\widetilde{f}}
\def\sgd{\on{sgd}}
\def\lfl{\lfloor}
\def\rfl{\rfloor}
\def\ord{\on{ord}}
\def\gd{{\on{gd}}}
\def\rk{{{\on{rk}}_2}}
\def\nbar{{\overline{n}}}
\def\MC{\on{MC}}
\def\lg{{\on{lg}}}
\def\cH{\mathcal{H}}
\def\cS{\mathcal{S}}
\def\cP{\mathcal{P}}
\def\N{{\Bbb N}}
\def\Z{{\Bbb Z}}
\def\Q{{\Bbb Q}}
\def\R{{\Bbb R}}
\def\C{{\Bbb C}}
\def\Lb{\overline\Lambda}
\def\mo{\on{mod}}
\def\xt{\times}
\def\notimm{\not\subseteq}
\def\Remark{\noindent{\it  Remark}}
\def\kut{\widetilde{KU}}
\def\Eb{\overline E}
\def\*#1{\mathbf{#1}}
\def\0{$\*0$}
\def\1{$\*1$}
\def\22{$(\*2,\*2)$}
\def\33{$(\*3,\*3)$}
\def\ss{\smallskip}
\def\ssum{\sum\limits}
\def\dsum{\displaystyle\sum}
\def\la{\langle}
\def\ra{\rangle}
\def\on{\operatorname}
\def\proj{\on{proj}}
\def\od{\text{od}}
\def\ev{\text{ev}}
\def\o{\on{o}}
\def\U{\on{U}}
\def\lg{\on{lg}}
\def\a{\alpha}
\def\bz{{\Bbb Z}}
\def\eps{\varepsilon}
\def\bc{{\mathbf C}}
\def\bN{{\mathbf N}}
\def\bB{{\mathbf B}}
\def\bW{{\mathbf W}}
\def\nut{\widetilde{\nu}}
\def\tfrac{\textstyle\frac}
\def\b{\beta}
\def\G{\Gamma}
\def\g{\gamma}
\def\zt{{\Bbb Z}_2}
\def\code#1{\texttt{#1}}
\def\zth{{\mathbf Z}_2^\wedge}
\def\bs{{\mathbf s}}
\def\bx{{\mathbf x}}
\def\bof{{\mathbf f}}
\def\bq{{\mathbf Q}}
\def\be{{\mathbf e}}
\def\lline{\rule{.6in}{.6pt}}
\def\xb{{\overline x}}
\def\xbar{{\overline x}}
\def\ybar{{\overline y}}
\def\zbar{{\overline z}}
\def\ebar{{\overline \be}}
\def\nbar{{\overline n}}
\def\ubar{{\overline u}}
\def\bbar{{\overline b}}
\def\et{{\widetilde e}}
\def\M{\mathcal{M}}
\def\lf{\lfloor}
\def\rf{\rfloor}
\def\ni{\noindent}
\def\ms{\medskip}
\def\Dhat{{\widehat D}}
\def\what{{\widehat w}}
\def\Yhat{{\widehat Y}}
\def\abar{{\overline{a}}}
\def\minp{\min\nolimits'}
\def\sb{{$\ssize\bullet$}}
\def\mul{\on{mul}}
\def\N{{\Bbb N}}
\def\Z{{\Bbb Z}}
\def\S{\Sigma}
\def\Q{{\Bbb Q}}
\def\R{{\Bbb R}}
\def\C{{\Bbb C}}
\def\Xb{\overline{X}}
\def\eb{\overline{e}}
\def\notint{\cancel\cap}
\def\cS{\mathcal S}
\def\cR{\mathcal R}
\def\el{\ell}
\def\TC{\on{TC}}
\def\GC{\on{GC}}
\def\wgt{\on{wgt}}
\def\Ht{\widetilde{H}}
\def\wbar{\overline w}
\def\dstyle{\displaystyle}
\def\Sq{\on{sq}}
\def\Om{\Omega}
\def\ds{\dstyle}
\def\tz{tikzpicture}
\def\zcl{\on{zcl}}
\def\bd{\mathbf{d}}
\def\cM{\mathcal{M}}
\def\io{\iota}
\def\Vb#1{{\overline{V_{#1}}}}
\def\Ebar{\overline{E}}
\def\lb{$\scriptstyle\bullet$}
\def\lb{\,\begin{picture}(-1,1)(1,-1)\circle*{4.5}\end{picture}\ }
\def\lbb{\,\begin{picture}(-1,1)(1,-1)\circle*{8}\end{picture}\ }
\def\zp{\Z_p}
\def\bL{\mathbf{L}}

\title
{Isomorphism classes of cut loci for a cube}
\author{Donald M. Davis}
\address{Department of Mathematics, Lehigh University\\Bethlehem, PA 18015, USA}
\email{dmd1@lehigh.edu}
\author{Manyi Guo}
\address{Department of Mathematics,  Lehigh University\\Bethlehem, PA 18015, USA}
\email{maga23@lehigh.edu}
\date{February 9, 2023}


\maketitle
\begin{minipage}{5.6in}
\begin{center}
     \begin{abstract}
    
    We prove that a face of a cube can be optimally partitioned into  193 connected sets on which the cut locus, or ridge tree, is constant up to isomorphism as a labeled graph. These are 60 connected open sets, curves bounding them, and intersection points of curves. Polynomial equations for the curves are given. Sixteen pairs of sets support the same cut locus class. We present the 177 distinct cut locus classes.

\end{abstract}
\end{center}

\end{minipage}

\section{Introduction}\label{intro}
The \textit{cut locus}, or ridge tree, of a point $P$ on a convex polyhdedron is the set of points $Q$ for which there is more than one shortest path from $P$ to $Q$. Each cut locus is a tree whose leaves are corner points\footnote{We reserve the term {\it vertex} for vertices of the star unfolding and cut locus.} of the polyhedron. In \cite{star97} and \cite{cut21}, methods were developed for determining the cut locus of a point, which we describe in Section $\ref{background}$ and utilize.

The cut locus of $P$ varies continuously with  $P$ unless $P$ is a corner point of the polyhedron, but its combinatorial structure can change abruptly. We think of a cut locus as a  graph with some vertices labeled by corner points of the polyhedron. We define an equivalence relation for these labeled graphs by edge-preserving vertex bijections that preserves labels, and denote by $\textbf{L}$ the equivalence class of a cut locus $L$.

In this paper, we consider the cut loci for a cube and find a complete decomposition of a face of a cube into connected subsets on which $\textbf{L}$ is constant. The subsets are connected open sets, curves bounding these sets, and single points where the curves intersect.
These are accurately rendered in Figure \ref{figA}; Figure \ref{figB} gives an expanded version of regions in the left quadrant of \ref{figA}. We find that there are $193$ connected subsets altogether, but of these there are 16 pairs which have the same $\bL$ and so there are 177 distinct $\bL$ on a face of a cube.

\bigskip

\begin{minipage}{6in}

\begin{fig} \label{figA}
 {\bf Decomposition of a face into subsets on which $\textbf{L}$ is constant}
\\
\bigskip
\begin{center}

    \includegraphics[scale=0.25]{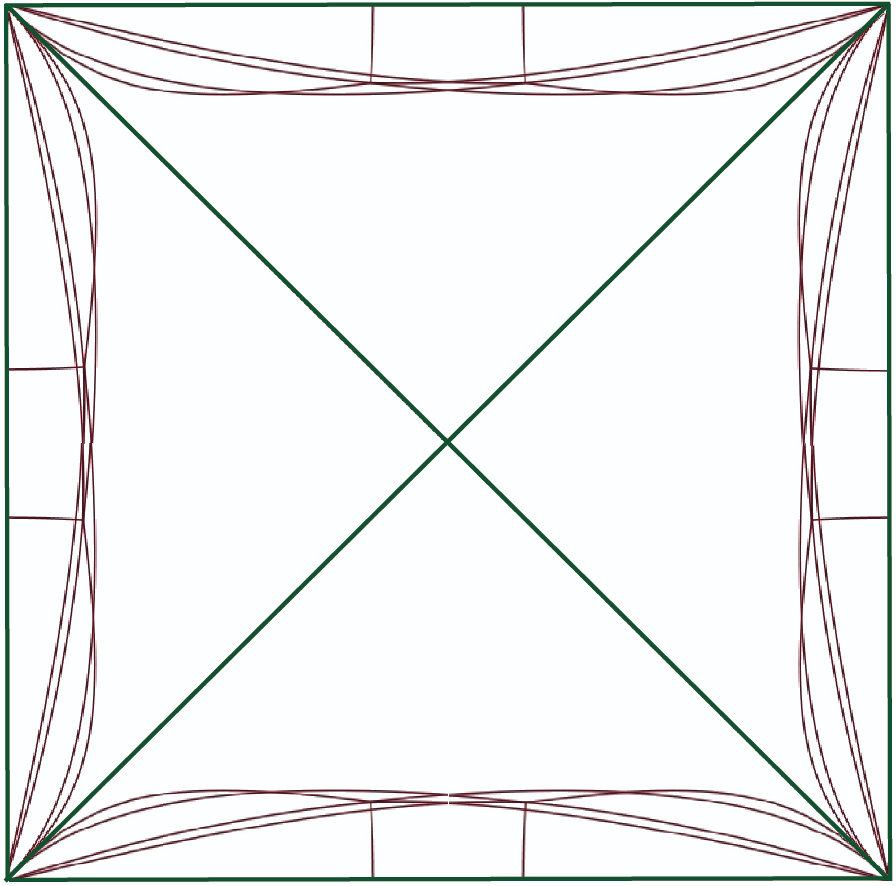}

\end{center}
\end{fig}
\end{minipage}
\\
\bigskip

In Section \ref{stmtsec}, we give a precise statement of results, including equations of the curves bounding the regions, and the labeled graphs for a representative set of $\bL$. In Section \ref{background}, we  present some preliminary information and tools from \cite{cut21} and \cite{star97} needed in our work. Section \ref{pfsec} sketches a proof that the regions and isomorphism classes of their cut loci are as described, and in Section \ref{endsec}, we complete the proof.

In Figure \ref{Cube}, we picture a  cube and a typical cut locus on it. This shows the numbering of the corner points of the cube that we will use throughout. We use the back face of that cube as the domain for our points $P$.

\bigskip

\begin{minipage}{6in}

\begin{fig} \label{Cube}

{\bf A cube with labeled corner points, and the cut locus for the middle point of an edge highlighted}
\\
\begin{center}

    \includegraphics[scale=0.18]{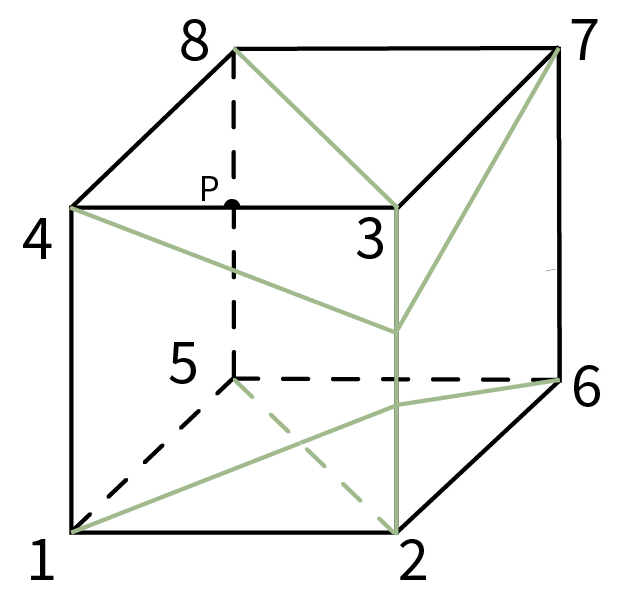}

\end{center}
\end{fig}
\end{minipage}

 \bigskip

One motivation for this work was \cite{david}, which considered geodesic motion-planning rules on a cube. Another was \cite{star97}, which considered bounds for the number of equivalence classes of cut loci on a convex polyhedron.

\section{Statement of results}\label{stmtsec}
In this section, we state our main result, an optimal decomposition of a face of a cube into 193 connected subsets on which the isomorphism class $\bL$ of the cut locus is constant, together with a depiction of these $bL$. Proofs of all claims will appear in Sections \ref{pfsec} and \ref{endsec}.
Because of the omnibus nature of our result, we do not organize it into ``Theorems.''

 We show that a face of the cube is composed of 60 connected open sets on which $\bL$ is constant, together with 48 curves which bound  these regions. Except for the boundary of the square and its diagonals, each of these curves is given by a 2-variable polynomial equation of degree 2 or 3 with integer coefficients. Some of these curves have constant $\bL$, while others are divided into two or three adjacent portions, on each of which $\bL$ is constant, yielding 96 curve portions with constant $\bL$. There are 58 distinct $\bL$'s on the regions and 86 on the curves. There are 37 points of intersection of these curves, giving 33 additional $\bL$.

We find it convenient to use $0\le x\le8$ and $-4\le y\le4$ as the coordinates of $P=(x,y)$ in our face. Figure \ref{figB} depicts the 15 open regions in the quadrant $Q_1=\{(x,y):0\le x\le4, |y|\le 4-x\}$. In Figure \ref{figB}, the $x$-axis is stretched by a factor of nearly 5 in order to better display the regions. Figure \ref{figA} depicts the whole square, illustrating how regions in the other three quadrants are copies of the regions in the quadrant $Q_1$ rotated around the center of the square. We will explain how the $\bL$ in the regions in these quadrants are obtained by permuting the corner numbers 1-8 in $\bL$.

\bigskip
\begin{minipage}{6in}
\begin{fig}\label{figB}

{\bf Regions in quadrant $Q_1$}

   \begin{center}
   
    \includegraphics[scale=0.09]{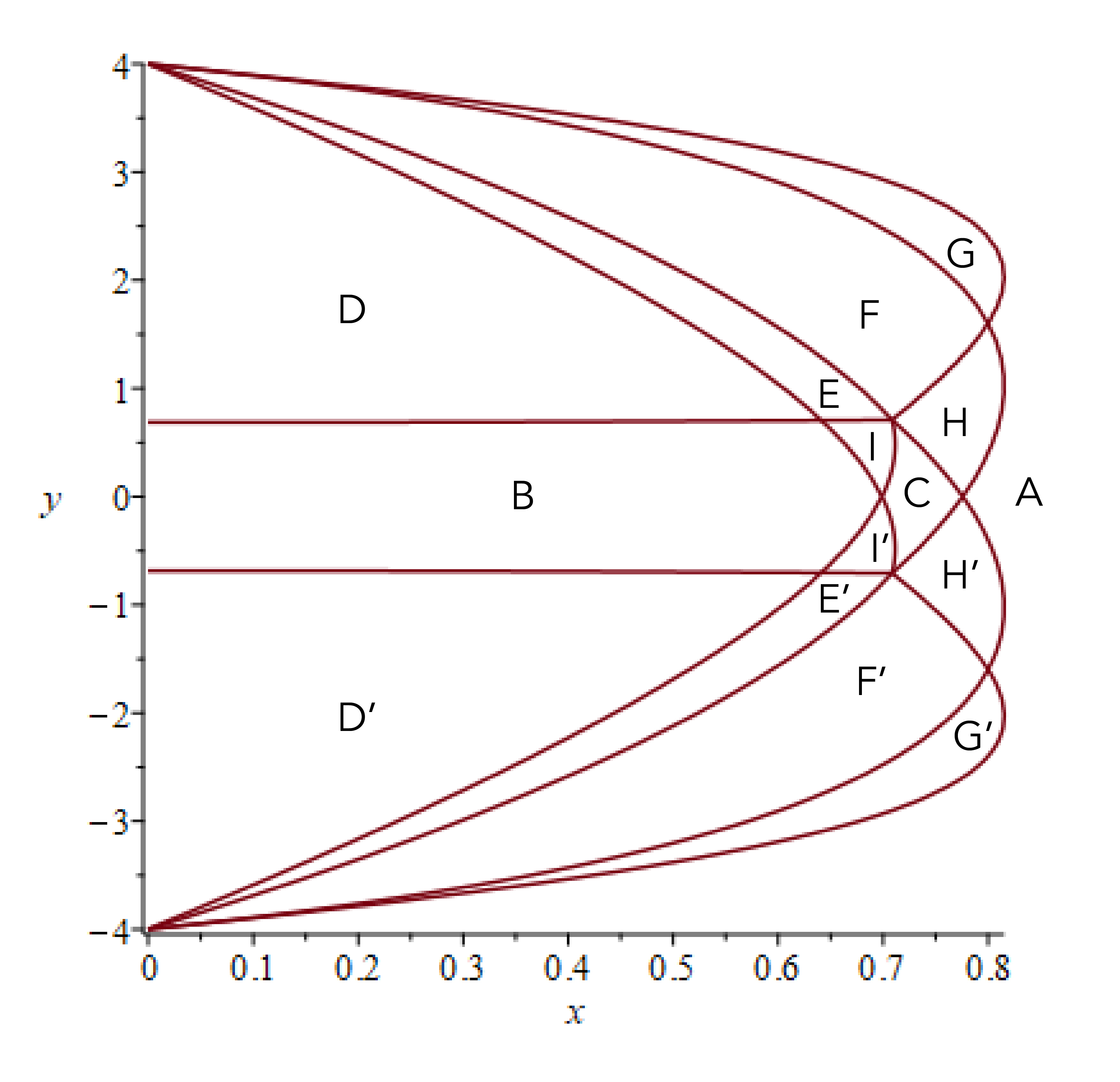}
    
\end{center}
    \end{fig}
\end{minipage}

\bigskip


In Figure \ref{figC}, we present the $\bL$ for points in regions $A$-$I$ in Figure \ref{figB}. The $\bL$ in the primed regions in Figure \ref{figB} are obtained by applying the permutation $\tau=(1\ 4)(2\ 3)(5\ 8)(6\ 7)$ to the corner numbers in the $\bL$ of the corresponding unprimed region $D$-$I$. Note that the graphs which appear in Figure \ref{figC} represent isomorphism classes of labeled graphs, and so whether an edge points to the left or right is irrelevant, as is the vertical orientation of the graph.

\bigskip
\begin{minipage}{6in}
\begin{fig}\label{figC}

{\bf $\bL$ in regions.}

\begin{center}

\begin{\tz}[scale=.5]
\draw (-2,-2) -- (-1,-2) -- (0,-1) -- (1,-2) -- (2,-2);
\draw (-1,-3) -- (-1,-2);
\draw (1,-3) -- (1,-2);
\draw (0,-1) -- (0,1) -- (1,2) -- (2,2);
\draw (1,2) -- (1,3);
\draw (0,1) -- (-1,2) -- (-2,2);
\draw (-1,2) -- (-1,3);
\draw (4,-3) -- (8,-3);
\draw (6,-3) -- (6,2);
\draw (4,2) -- (8,2);
\draw (6,-2) -- (8,-2);
\draw (6,-1) -- (8,-1);
\draw (6,0) -- (8,0);
\draw (6,1) -- (8,1);
\draw (10,-3) -- (14,-3);
\draw (12,-3) -- (12,2);
\draw (10,2) -- (14,2);
\draw (12,-2) -- (14,-2);
\draw (12,-1) -- (14,-1);
\draw (12,0) -- (14,0);
\draw (12,1) -- (14,1);
\draw (16,-3) -- (20,-3);
\draw (18,-3) -- (18,2);
\draw (16,2) -- (20,2);
\draw (18,-2) -- (20,-2);
\draw (18,-1) -- (20,-1);
\draw (18,0) -- (20,0);
\draw (18,1) -- (20,1);
\draw (22,-3) -- (26,-3);
\draw (24,-3) -- (24,2);
\draw (22,2) -- (26,2);
\draw (24,-2) -- (26,-2);
\draw (24,-1) -- (26,-1);
\draw (24,0) -- (26,0);
\draw (24,1) -- (26,1);
\node at (8.2,2) {$2$};
\node at (8.2,1) {$6$};
\node at (9.8,2) {$2$};
\node at (14.2,1) {$5$};
\node at (8.2,0) {$1$};
\node at (14.2,0) {$1$};
\node at (8.2,-1) {$4$};
\node at (8.2,-2) {$7$};
\node at (8.2,-3) {$3$};
\node at (3.8,2) {$5$};
\node at (3.8,-3) {$8$};
\node at (14.2,-1) {$4$};
\node at (14.2,-2) {$8$};
\node at (14.2,-3) {$7$};
\node at (9.8,-3) {$3$};
\node at (14.2,2) {$6$};
\node at (15.8,2) {$5$};
\node at (15.8,-3) {$8$};
\node at (20.2,2) {$2$};
\node at (20.2,1) {$6$};
\node at (20.2,0) {$1$};
\node at (20.2,-1) {$7$};
\node at (20.2,-2) {$4$};
\node at (20.2,-3) {$3$};
\node at (21.8,2) {$2$};
\node at (21.8,-3) {$8$};
\node at (26.2,2) {$6$};
\node at (26.2,1) {$5$};
\node at (26.2,0) {$1$};
\node at (26.2,-1) {$7$};
\node at (26.2,-2) {$4$};
\node at (26.2,-3) {$3$};
\node at (2.2,2) {$6$};
\node at (2.2,-2) {$7$};
\node at (-2.2,2) {$1$};
\node at (-2.2,-2) {$4$};
\node at (-1,3.2) {$5$};
\node at (1,3.2) {$2$};
\node at (-1,-3.2) {$8$};
\node at (1,-3.2) {$3$};
\node at (0,3.8) {$A$};
\node at (6,3) {$B$};
\node at (12,3) {$C$};
\node at (18,3) {$D$};
\node at (24,3) {$E$};
\draw (-2,-10) -- (2,-10);
\draw (0,-10) -- (0,-5);
\draw (0,-9) -- (2,-9);
\draw (0,-8) -- (2,-8);
\draw (0,-7) -- (2,-7);
\draw (0,-6) -- (2,-6);
\draw (-2,-5) -- (2,-5);
\draw (6,-10) -- (10,-10);
\draw (8,-9) -- (10,-9);
\draw (8,-8) -- (10,-8);
\draw (8,-10) -- (8,-7) -- (9,-6) -- (10,-6);
\draw (9,-6) -- (9,-5);
\draw (8,-7) -- (7,-6) -- (6,-6);
\draw (7,-6) -- (7,-5);
\draw (14,-5) -- (18,-5);
\draw (16,-6) -- (18,-6);
\draw (16,-7) -- (18,-7);
\draw (16,-5) -- (16,-8) -- (17,-9) -- (18,-9);
\draw (17,-9) -- (17,-10);
\draw (16,-8) -- (15,-9) -- (14,-9);
\draw (15,-9) -- (15,-10);
\draw (22,-10) -- (26,-10);
\draw (24,-10) -- (24,-5);
\draw (22,-5) -- (26,-5);
\draw (24,-9) -- (26,-9);
\draw (24,-8) -- (26,-8);
\draw (24,-7) -- (26,-7);
\draw (24,-6) -- (26,-6);
\node at (-2.2,-5) {$2$};
\node at (-2.2,-10) {$4$};
\node at (2.2,-5) {$6$};
\node at (2.2,-6) {$5$};
\node at (2.2,-7) {$1$};
\node at (2.2,-8) {$7$};
\node at (2.2,-9) {$3$};
\node at (2.2,-10) {$8$};
\node at (5.8,-10) {$4$};
\node at (10.2,-10) {$8$};
\node at (10.2,-9) {$3$};
\node at (10.2,-8) {$7$};
\node at (10.2,-6) {$6$};
\node at (5.8,-6) {$1$};
\node at (9,-4.8) {$2$};
\node at (7,-4.8) {$5$};
\node at (13.8,-5) {$2$};
\node at (13.8,-9) {$4$};
\node at (18.2,-5) {$6$};
\node at (18.2,-6) {$5$};
\node at (18.2,-7) {$1$};
\node at (18.2,-9) {$7$};
\node at (15,-10) {$8$};
\node at (17,-10) {$3$};
\node at (21.8,-5) {$2$};
\node at (21.8,-10) {$8$};
\node at (26.2,-10) {$3$};
\node at (26.2,-9) {$7$};
\node at (26.2,-8) {$4$};
\node at (26.2,-7) {$1$};
\node at (26.2,-6) {$5$};
\node at (26.2,-5) {$6$};
\node at (0,-11) {$F$};
\node at (8,-11) {$G$};
\node at (16,-11) {$H$};
\node at (24,-11) {$I$};
\end{\tz}
\end{center}
\end{fig}
\end{minipage}
\bigskip

Each region $R$ in the top quadrant in Figure \ref{figA} is obtained from the corresponding region $R_0$ in quadrant $Q_1$ by a clockwise rotation of $\pi/2$ around the center of the square. The $\bL$ for $R$ is obtained from that of $R_0$ by applying the permutation $\sigma=(1\ 4\ 3\ 2)(5\ 8\ 7\ 6)$ to the corner numbers at vertices. Similarly, regions along the right edge are a $\pi$-rotation of $R_0$ and have their $\bL$ obtained using the permutation $\sigma^2=(1\ 3)(2\ 4)(5\ 7)(6\ 8)$. Finally, a clockwise rotation of $3\pi/2$ applies $\sigma^3=(1\ 2\ 3\ 4)(5\ 6\ 7\ 8)$ to the numbers at vertices of $\bL$. One can check that, for the 15 regions $R_0$ in $Q_1$, the $\bL$ for $\sigma^iR_0$, $0\le i\le 3$, are distinct except that $\s^{2+\eps}\bL_A=\s^\eps\bL_A$ for $\eps=0,1$, yielding 58 distinct $\bL$ for the regions on the face, each $\bL$ having six degree-3 vertices. The notation $\bL_A$ refers to the $\bL$ of points in the region $A$.

There are five curves and their vertical reflections which bound pairs of regions in Figure \ref{figB}. A single curve usually bounds more than one pair of regions. Its $\bL$ will be different for different pairs. In every case, the $\bL$ for the curve portion is obtained by collapsing to a point one edge of the $\bL$ for each region which it bounds.

For each curve, we list in (\ref{five}) the pairs of regions which it bounds, followed by its equation. Then in Figure \ref{curveL}, we present the $\bL$ for the various curve portions. For example, the  first curve, which appears almost horizontal in Figure \ref{figA} but is actually an arc of a circle with large radius, bounds regions B and D, and then has a short portion bounding regions E and I, and its $\bL$ for each of these portions is presented in Figure \ref{curveL}. The intersection point of these two portions has a different $\bL$, which will be described, along with its coordinates, later in this section.

\begin{align}BD,EI&\quad x^2+y^2-24y+16=0\label{five}\\
DE,BI,CI'&\quad y^3+(3x+12)y^2+(x^2+40x-16)y+3x^3-44x^2+304x-192=0\nonumber\\
EF&\quad y^3+(x-12)y^2+(x^2+8x-16)y+x^3-20x^2-240x+192=0\nonumber\\
FG,HA,CH'&\quad x^3-4x^2+(y^2+8y-80)x-4y^2+64=0\nonumber\\
GA,FH&\quad x^3-12x^2+(y^2-24y+112)x+4y^2-64=0\nonumber\end{align}

\bigskip
\begin{minipage}{6in}
\begin{fig}\label{curveL}

{\bf $\bL$ on curves.}

\begin{center}

\begin{\tz}[scale=.5]
\draw (-1.5,0) -- (1.5,0);
\draw (0,0) -- (0,4);
\draw (-1.5,1) -- (1.5,1);
\draw (0,2) -- (1.5,2);
\draw (0,3) -- (1.5,3);
\draw (-1.5,4) -- (1.5,4);
\node at (-1.7,0) {$8$};
\node at (-1.7,1) {$4$};
\node at (-1.7,4) {$5$};
\node at (1.7,0) {$3$};
\node at (1.7,1){$7$};
\node at (1.7,2) {$1$};
\node at (1.7,3) {$6$};
\node at (1.7,4) {$2$};
\draw (3.5,0) -- (6.5,0);
\draw (3.5,1) -- (6.5,1);
\draw (3.5,4) -- (6.5,4);
\draw (5,2) -- (6.5,2);
\draw (5,3) -- (6.5,3);
\draw (5,0) -- (5,4);
\node at (3.3,0) {$8$};
\node at (3.3,1) {$4$};
\node at (3.3,4) {$2$};
\node at (6.7,0) {$3$};
\node at (6.7,1) {$7$};
\node at (6.7,2) {$1$};
\node at (6.7,3) {$5$};
\node at (6.7,4) {$6$};
\node at (0,5) {$BD$};
\node at (5,5) {$EI$};
\draw (8.5,0) -- (11.5,0);
\draw (8.5,4) -- (11.5,4);
\draw (10,1) -- (11.5,1);
\draw (10,2) -- (11.5,2);
\draw (10,3) -- (11.5,3);
\draw (10,0) -- (10,4.9);
\node at (11.2,5.7) {$DE$};
\node at (8.3,0) {$8$};
\node at (8.3,4) {$5$};
\node at (11.7,0){$3$};
\node at (11.7,1) {$4$};
\node at (11.7,2) {$7$};
\node at (11.7,3) {$1$};
\node at (11.7,4) {$6$};
\node at (10,5.1) {$2$};
\draw (13.5,0) -- (16.5,0);
\draw (13.5,4) -- (16.5,4);
\draw (15,1) -- (16.5,1);
\draw (15,2) -- (16.5,2);
\draw (15,3) -- (16.5,3);
\draw (15,0) -- (15,4.9);
\node at (16.2,5.7) {$BI$};
\node at (13.3,0) {$8$};
\node at (13.3,4) {$5$};
\node at (16.7,0){$3$};
\node at (16.7,1) {$7$};
\node at (16.7,2) {$4$};
\node at (16.7,3) {$1$};
\node at (16.7,4) {$6$};
\node at (15,5.1) {$2$};
\draw (18.5,0) -- (21.5,0);
\draw (18.5,4) -- (21.5,4);
\draw (20,1) -- (21.5,1);
\draw (20,2) -- (21.5,2);
\draw (20,3) -- (21.5,3);
\draw (20,0) -- (20,4.9);
\node at (21.2,5.7) {$CI'$};
\node at (18.3,0) {$3$};
\node at (18.3,4) {$5$};
\node at (21.7,0){$7$};
\node at (21.7,1) {$8$};
\node at (21.7,2) {$4$};
\node at (21.7,3) {$1$};
\node at (21.7,4) {$6$};
\node at (20,5.1) {$2$};
\draw (23.5,0) -- (26.5,0);
\draw (25,1) -- (26.5,1);
\draw (23.5,4) -- (26.5,4);
\draw (25,2) -- (26.5,2);
\draw (25,3) -- (26.5,3);
\draw (25,-.7) -- (25,4);  
\node at (23.3,0) {$4$};
\node at (23.3,4) {$2$};
\node at (26.7,0) {$3$};
\node at (26.7,1) {$7$};
\node at (26.7,2) {$1$};
\node at (26.7,3) {$5$};
\node at (26.7,4) {$6$};
\node at (25,-.9) {$8$};
\node at (25,5) {$EF$};
\draw (-1.5,-6) -- (1.5,-6);
\draw (-1.5,-3) -- (1.5,-3);
\draw (-1.5,-2) -- (1.5,-2);
\draw (0,-5) -- (1.5,-5);
\draw (0,-4) -- (1.5,-4);
\draw (0,-6) -- (0,-2);
\node at (-1.7,-6) {$4$};
\node at (-1.7,-3) {$1$};
\node at (-1.7,-2) {$2$};
\node at (1.7,-6) {$8$};
\node at (1.7,-5) {$3$};
\node at (1.7,-4) {$7$};
\node at (1.7,-3) {$5$};
\node at (1.7,-2) {$6$};
\node at (0,-7) {$FG$};
\draw (5,-6) -- (5,-5) -- (6,-4) -- (6,-2);
\draw (4,-5) -- (5,-5);
\draw (7,-6) -- (7,-5);
\draw (8,-5) -- (7,-5) -- (6,-4);
\draw (4.5,-3) -- (7.5,-3);
\draw (4.5,-2) -- (7.5,-2);
\node at (3.8,-5) {$4$};
\node at (8.2,-5) {$7$};
\node at (5,-6.2) {$8$};
\node at (7,-6.2) {$3$};
\node at (4.3,-3) {$1$};
\node at (4.3,-2) {$2$};
\node at (7.7,-3) {$5$};
\node at (7.7,-2) {$6$};
\node at (6,-7.2) {$HA$};
\draw (10.5,-6) -- (13.5,-6);
\draw (10.5,-3) -- (13.5,-3);
\draw (10.5,-2) -- (13.5,-2);
\draw (12,-5) -- (13.5,-5);
\draw (12,-4) -- (13.5,-4);
\draw (12,-6) -- (12,-2);
\node at (10.3,-6) {$3$};
\node at (10.3,-3) {$1$};
\node at (10.3,-2) {$2$};
\node at (13.7,-6) {$7$};
\node at (13.7,-5) {$8$};
\node at (13.7,-4) {$4$};
\node at (13.7,-3) {$5$};
\node at (13.7,-2) {$6$};
\node at (12,-7) {$CH'$};
\node at (18,-7) {$GA$};
\draw (18,-6) -- (18,-4) -- (19,-3) -- (20,-3);
\draw (16.5,-6) -- (19.5,-6);
\draw (16.5,-5) -- (19.5,-5);
\draw (19,-3) -- (19,-2);
\draw (18,-4) -- (17,-3) -- (16,-3);
\draw (17,-3) -- (17,-2);
\node at (16.3,-6) {$4$};
\node at (16.3,-5) {$3$};
\node at (19.7,-6) {$8$};
\node at (19.7,-5) {$7$};
\node at (15.8,-3) {$1$};
\node at (20.2,-3) {$6$};
\node at (17,-1.8) {$5$};
\node at (19,-1.8) {$2$};
\draw (22.5,-6) -- (25.5,-6);
\draw (24,-3) -- (25.5,-3);
\draw (22.5,-2) -- (25.5,-2);
\draw (22.5,-5) -- (25.5,-5);
\draw (24,-4) -- (25.5,-4);
\draw (24,-6) -- (24,-2);
\node at (22.3,-6) {$4$};
\node at (25.7,-4) {$1$};
\node at (22.3,-2) {$2$};
\node at (25.7,-6) {$8$};
\node at (22.3,-5) {$3$};
\node at (25.7,-5) {$7$};
\node at (25.7,-3) {$5$};
\node at (25.7,-2) {$6$};
\node at (24,-7) {$FH$};

\end{\tz}
\end{center}
\end{fig}
\end{minipage}

The vertical reflection of the curves is obtained by replacing $y$ by $-y$ in the equations, and their $\bL$ is obtained using the permutation $\tau=(1\ 4)(2\ 3)(5\ 8)(6\ 7)$, as before.
For the other three quadrants, the equations can be modified in an obvious way, and the $\bL$ obtained using the same permutations as were used for regions. One can check that, for the 11 curve segments $S$ in Figure \ref{curveL} and for $\eps=0,1$ and $0\le i\le3$, the $\bL$ for $\tau^\eps\s^iS$ are distinct except that $\tau^\eps\s^{2+i}\bL_{GA}=\tau^\eps\s^i\bL_{HA}$. This gives $88-8$ distinct $\bL$'s for curve portions. All of these $\bL$'s have five degree-3 vertices.

In addition, there are 6 more $\bL$'s, coming from the edges and half-diagonals of the square. The entire left edge of our face has  constant $\bL$, as does the half diagonal $h$ connecting the center of the face with the upper left corner. These are shown in Figure \ref{two}. These are the first cases where a corner point does not appear at a leaf, but rather at a degree-2 vertex of the cut locus. Applying the permutations $\sigma$, $\sigma^2$, and $\sigma^3$ described earlier gives the $\bL$'s on the other edges and half diagonals. However, $\s^{2+\eps}\bL_h=\s^\eps\bL_h$. Combining with those described above yields 86 distinct $\bL$'s on portions of curves.

\bigskip
\begin{minipage}{6in}
\begin{fig}\label{two}

{\bf $\bL$ on left edge and upper-left half diagonal.}

\begin{center}

\begin{\tz}[scale=.5]
\draw (-1.5,1) -- (0,1) -- (0,4) -- (-1.5,4);
\draw (-1.5,3) -- (1.5,3);
\draw (-1.5,2) -- (1.5,2);
\node at (-1.7,1) {$8$};
\node at (-1.7,2) {$4$};
\node at (-1.7,3) {$1$};
\node at (-1.7,4) {$5$};
\node at (1.7,2) {$7$};
\node at (1.7,3) {$6$};
\node at (.2,.8) {$3$};
\node at (.2,4.2) {$2$};
\draw (9,0) -- (9,1) -- (10,2.5) -- (11,4) -- (12,4);
\draw (11,0) -- (11,1) -- (12,1);
\draw (10,2.5) -- (11,1);
\draw (10,2.5) -- (9,4) -- (9,5);
\draw (9,4) -- (8,4);
\node at (9,-.2) {$8$};
\node at (11,-.2) {$3$};
\node at (12.2,1) {$7$};
\node at (12.2,4) {$6$};
\node at (7.8,4) {$1$};
\node at (9,5.2) {$5$};`
\node at (11,4.3) {$2$};
\node at (8.8,1.2) {$4$};
 \node at (-.3,-1.3) {\rm left edge};
 \node at (10,-1.3) {\rm half diagonal};
\end{\tz}
   
\end{center}
\end{fig}
\end{minipage}
\bigskip

Not including the $y$-axis, Figure \ref{figB} has eight intersection points of curves. Three below the $x$-axis are obtained by vertical reflection, and their $\bL$ is obtained using the usual permutation $\tau$. We list the other five, denoting them by the regions  abutting them, and include their coordinates. 
\begin{eqnarray*}BDEI&(0.6413,0.7045)\\
EFHCI&(0.7085,0.7085)\\
FGHA&(0.8,1.6)\\
BII'C&(0.6989,0)\\
CHH'A&(0.7757,0)\end{eqnarray*}
More precisely, the $0.7085$ is $6-2\sqrt7$, while the $0.7045$, $0.6989$, and $0.7757$ are roots of the polynomials $37y^4-816y^3+304y^2-3456y+2560$, $3x^3-44x^2+304x-192$, and $x^3-4x^2-80x+64$, respectively. The $\bL$ for the five vertices are shown in Figure \ref{pts}. The $BD$ curve intersects the $y$-axis at $(0,0.685)$, but this point does not give a new $\bL$, since $\bL$ is constant on the $y$-axis (for $|y|<4$). 

\bigskip
\begin{minipage}{6in}
\begin{fig}\label{pts}

{\bf $\bL$ for intersection points.}

\begin{center}

\begin{\tz}[scale=.5]
\draw (-1.5,0) -- (1.5,0);
\draw (-1.5,1) -- (1.5,1);
\draw (0,2) -- (1.5,2);
\draw (-1.5,3) -- (1.5,3);
\draw (0,0) -- (0,4);
\node at (-1.7,0) {$8$};
\node at (-1.7,1) {$4$};
\node at (-1.7,3) {$5$};
\node at (1.7,0) {$3$};
\node at (1.7,1) {$7$};
\node at (1.7,2) {$1$};
\node at (1.7,3) {$6$};
\node at (0,4.2) {$2$};
\node at (0,-1.2) {$BDEI$};
\draw (4.5,1) -- (7.5,1);
\draw (6,2) -- (7.5,2);
\draw (6,3) -- (7.5,3);
\draw (4.5,4) -- (7.5,4);
\draw (6,1) -- (6,4);
\draw (5,0) -- (6,1) -- (7,0);
\node at  (4.3,1) {$4$};
\node at (4.3,4) {$2$};
\node at (7.7,1) {$7$};
\node at (7.7,2) {$1$};
\node at (7.7,3) {$5$};
\node at (7.7,4) {$6$};
\node at (5,0) {$8$};
\node at (7,0) {$3$};
\node at (6,-1.2) {$EFHCI$};
\draw (10.5,0) -- (13.5,0);
\draw (10.5,1) -- (13.5,1);
\draw (10.5,2) -- (13.5,2);
\draw (10.5,3) -- (13.5,3);
\draw (12,0) -- (12,3);
\node at (10.3,0) {$4$};
\node at (10.3,1) {$3$};
\node at (10.3,2) {$1$};
\node at (10.3,3) {$2$};
\node at (13.7,0) {$8$};
\node at (13.7,1) {$7$};
\node at (13.7,2) {$5$};
\node at (13.7,3) {$6$};
\node at (12,-1.2) {$FGHA$};
\draw (22.5,0) -- (25.5,0);
\draw (22.5,1) -- (25.5,1);
\draw (22.5,2) -- (25.5,2);
\draw (22.5,3) -- (25.5,3);
\draw (24,0) -- (24,3);
\node at (22.3,0) {$3$};
\node at (22.3,1) {$4$};
\node at (22.3,2) {$1$};
\node at (22.3,3) {$2$};
\node at (25.7,0) {$7$};
\node at (25.7,1) {$8$};
\node at (25.7,2) {$5$};
\node at (25.7,3) {$6$};
\node at (24,-1.2) {$CHH'A$};
\draw (16.5,0) -- (19.5,0);
\draw (18,1) -- (19.5,1);
\draw (18,2) -- (19.5,2);
\draw (16.5,3) -- (19.5,3);
\draw (18,-.9) -- (18,4);
\node at (16.3,0) {$8$};
\node at (16.3,3) {$5$};
\node at (19.7,0) {$7$};
\node at (19.7,1) {$4$};
\node at (19.7,2) {$1$};
\node at (19.7,3) {$6$};
\node at (18,4.2) {$2$};
\node at (18,-1.1) {$3$};
\node at (20,-1.3) {$BII'C$};
\end{\tz}
    \end{center}
    \end{fig}
    \end{minipage}

    \bigskip
Including $\tau\bL$ for the first three $\bL$ in Figure \ref{pts} and applying $\s^i$, $0\le i\le3$ to all gives $32-4$ distinct $\bL$, as $\tau^\eps\s^{2+i}\bL_{FGHA}=\tau^\eps\s^i\bL_{FGHA}$.
    
Finally, there are vertices at the center of the face and at each corner. The $\bL$ for the center and the top-left corner are presented in Figure \ref{final}. Those for the other corners are obtained using the usual permutations.    
    
\bigskip
\begin{minipage}{6in}
\begin{fig}\label{final}

{\bf $\bL$ for special points.}

\begin{center}

\begin{\tz}[scale=.6]
\draw (0,0) -- (0,1) -- (2,3) -- (3,3);
\draw (3,1) -- (2,1) -- (0,3) -- (0,4);
\node at (0,-.2) {$8$};
\node at (-.3,1) {$4$};
\node at (3.2,1) {$7$};
\node at (2,.7) {$3$};
\node at (3.2,3) {$6$};
\node at (2,3.3) {$2$};
\node at (-.3,3) {$1$};
\node at (0,4.2) {$5$};
\node at (1.5,-1) {\rm center};
\draw (6,1) -- (7,2) -- (8,2);
\draw (7,1) -- (7,2) -- (6,3);
\draw (8,1) -- (7,2) -- (6,2);
\node at (6,.7) {$4$};
\node at (7,.7) {$3$};
\node at (8,.7) {$7$};
\node at (5.7,2) {$1$};
\node at (5.7,3) {$5$};
\node at (8.3,2) {$6$};
\node at (7.2,2.3) {$2$};
\node at (7,-1) {\rm corner};
\end{\tz}
    \end{center}
\end{fig}
\end{minipage}

\section{Background}\label{background}

In this section we explain how the method for finding cut loci of convex polyhedra developed in \cite{star97} and \cite{cut21} applies to a cube. This involves star unfolding and Voronoi diagrams.

We consider the cube with corner points numbered as in Figure \ref{Cube}, and let $P$ be a point on the back (5678) face. In a planar model $M$ of all faces of the  cube except the front (1234) face, choose a shortest path connecting $P$ to each corner point. These are called {\it cuts}. See Figure \ref{cuts}.

\bigskip

\begin{minipage}{6in}
\begin{fig} \label{cuts}
{\bf Example of cuts with respect to $P$}
\begin{center}
   
\begin{tikzpicture}[scale=0.7]

\draw[gray, thick] (-3,1) -- (3,1);
\draw[gray, thick] (1,-3) -- (1,3);
\draw[gray, thick] (-1,-3) -- (-1,3);
\draw[gray, thick] (-3,-1) -- (3,-1);
\draw[gray, thick] (-3,-1) -- (-3,1);
\draw[gray, thick] (3,-1) -- (3,1);
\draw[gray, thick] (-1,3) -- (1,3);
\draw[gray, thick] (-1,-3) -- (1,-3);
\filldraw (-0.75,0.25) circle[radius=0.5pt];
\coordinate[label={[label distance=-3pt] $3$}] (C) at (1.25,3.25);
\coordinate[label={$2$}] (B) at (3.25,-1.5);
\coordinate[label={$1$}] (A) at (-1.25,-3.5);
\coordinate[label={[label distance=-3pt] $4$}] (D) at (-3.25,1.25);
\coordinate[label={[label distance=-3pt] $3$}] (C') at (3.25,1.25);
\coordinate[label={$2$}] (B') at (1.25,-3.5);
\coordinate[label={$1$}] (A') at (-3.25,-1.5);
\coordinate[label={ [label distance=-3pt] $4$}] (D') at (-1.25,3.25);
\coordinate[label={$P$}]  (P) at (-0.75,0.25);
\coordinate[label={$5$}]  (E) at (-1.25,-1.65);
\coordinate[label={$8$}]  (F) at (-1.25,1.25);
\coordinate[label={$6$}]  (G) at (1.25,-1.65);
\coordinate[label={$7$}]  (H) at (1.25,1.25);
\draw[gray,thin] (-3,1) -- (P);
\draw[gray,thin] (-3,-1) -- (P);
\draw[gray,thin] (1,3) -- (P);
\draw[gray,thin] (1,-1) -- (P);
\draw[gray,thin] (-1,1) -- (P);
\draw[gray,thin] (1,1) -- (P);
\draw[gray,thin] (-1,-1) -- (P);
\draw[gray,thin] (1,-3) -- (P);
\end{tikzpicture}

\end{center}
\end{fig}
\end{minipage}
\bigskip

These decompose $M$ as the union of eight polygons, with $P$ at a vertex of each, and edges 12, 23, 34, 41 and 15, 26, 37, and 48 at  far ends of the polygons. The {\it star unfolding} of $P$ is obtained by first gluing to the 1234 square the polygons with far edges 12, 23, 34, and 41. This will expose new edges 15, 26, 37, and 48, and we then glue the other four polygons to the corresponding edges. See Figure \ref{star}.
This yields a polygon with eight vertices corresponding to corner points of the cube, and eight corresponding to occurrences of the point $P$, which we number as in Figure \ref{star}. This is the star unfolding, $S$, of the point $P$.

\bigskip
\begin{minipage}{6in}
\begin{fig}\label{star}
{\bf A star unfolding $S$}
\begin{center}
    \begin{tikzpicture}[scale=0.7]
    \draw[gray, thick] (1,1) -- (-1,1);
    \draw[gray, thick] (-1,1) -- (-1,-1);
    \draw[gray, thick] (1,-1) -- (-1,-1);
    \draw[gray, thick] (1,1) -- (1,-1);
     \draw[gray, thick] (-1,1) -- (-2.25,3.25);
     \draw[gray, thick] (-1,1) -- (-3.25,-0.25);
     \draw[gray, thick] (-3.25,-0.25)--(-1,-1);
     \draw[gray, thick] (-1,-1)--(-1.75,-3.25);
      \draw[gray, thick] (-2.25,3.25)--(-1,3);
      \draw[gray, thick] (-1,3)--(-0.75,4.25);
      \draw[gray, thick] (-0.75,4.25)--(1,1);
      \draw [gray, thick,dash dot] (-1,-3) -- (-1,-1);
      \draw [gray, thick,dash dot] (-1,3) -- (-1,1);
       \draw[gray, thick] (-1.75,-3.25)--(-1,-3);
      \draw[gray, thick] (-1,-3)--(-0.75,-3.75);
      \draw[gray, thick] (-0.75,-3.75)--(1,-1);
      \draw[gray, thick] (1,-1)--(3.75,-2.75);
       \draw[gray, thick] (3.75,-2.75)--(3,-1);
       \draw[gray, thick] (3,-1)--(4.75,-0.25);
       \draw[gray, thick] (3,1)--(4.75,-0.25);
       \draw [gray, thick,dash dot] (1,1) -- (3,1);
       
       \draw [gray, thick,dash dot] (1,-1) -- (3,-1);
       \draw[gray, thick] (3,1)--(4.25,2.75);
       \draw[gray, thick] (4.25,2.75)--(1,1);
       \coordinate[label={$1$}]  (1) at (-1.5,0.75);
       \coordinate[label={$2$}]  (2) at (1.25,1.25);
       \coordinate[label={$3$}]  (3) at (1,-1.75);
       \coordinate[label={$4$}]  (4) at (-1.5,-1.5);
       \coordinate[label={$5$}]  (5) at (-1.25,3.25);
       \coordinate[label={$6$}]  (6) at (3.5,0.75);
       \coordinate[label={$7$}]  (7) at (3.5,-1.5);
       \coordinate[label={$8$}]  (8) at (-1.25,-3.75);
       \coordinate[label={$P_1$}]  (p_1) at (-3.75,-0.75);
       \coordinate[label={$P_2$}]  (p_2) at (-2.75,3);
       \coordinate[label={$P_3$}]  (p_3) at (-0.75,4.15);
       \coordinate[label={$P_4$}]  (p_4) at (4.5,2.75);
        \coordinate[label={$P_5$}]  (p_5) at (5.25,-0.75);
         \coordinate[label={$P_6$}]  (p_6) at (4,-3.5);
          \coordinate[label={$P_7$}]  (p_7) at (-0.75,-4.75);
           \coordinate[label={$P_8$}]  (p_3) at (-2.25,-3.5);
\end{tikzpicture}
\end{center}
\end{fig}
\end{minipage}

\bigskip
Recall that our coordinates for the 5678 face of the cube are $0\le x\le 8$ and $-4\le y\le 4$. We will initially consider points $P$ in the quadrant $Q_1$ given by $-4\le y\le 4$ and $0\le x\le 4-|y|$. Points in other quadrants will be considered later by rotating the cube.

We use $(v,w)$ as the coordinate system for the plane containing $S$, with $(0,0)$ at the midpoint of segment 2-3 in Figure \ref{star}, and sides of the two squares having length 8. The coordinates of the points labeled 1-8 are, respectively, $(-8,4)$, $(0,4)$, $(0,-4)$, $(-8,-4)$, $(-8,12)$, $(8,4)$, $(8,-4)$, and $(-8,-12)$. The coordinates
$(v_\a,w_\a)$ of the points $P_\a$ are as in (\ref{eqsvw}).

\begin{align}
     P_1 &= (-16-x,-y), &P_5&=(16-x,-y),\nonumber\\
P_2 &=(-12-y,12+x), &P_6&=(12-y,-12+x),\label{eqsvw}\\
P_3&=(-8+x,16+y), &P_7&=(-8+x,-16+y),\nonumber\\
P_4 &= (12+y,12-x), &P_8&=(-12+y,-12-x).\nonumber
\end{align}

For each point $P_\a$, $1\le\a\le8$, its {\it Voronoi cell} $C_\a$ is the set of points $Q$ in $S$ such that 
$$d(Q,P_\a)\le d(Q,P_\b)$$
for $1\le\b\le8$. The points of $S$ which lie in more than one $C_\a$ comprise the cut locus $L_P$ of $P$. In Figure \ref{cutloc}, we show the Voronoi cells and cut locus of the point $P$ in Figure \ref{cuts}.
\bigskip

\begin{minipage}{6in}
\begin{fig} \label{cutloc}
{\bf Voronoi cells and cut locus of $P$}
\begin{center}

    \begin{tikzpicture}[scale=0.7]
    \draw[gray, thick] (1,1) -- (-1,1);
    \draw[gray, thick] (-1,1) -- (-1,-1);
    \draw[gray, thick] (1,-1) -- (-1,-1);
    \draw[gray, thick] (1,1) -- (1,-1);
     \draw[gray, thick] (-1,1) -- (-2.25,3.25);
     \draw[gray, thick] (-1,1) -- (-3.25,-0.25);
     \draw[gray, thick] (-3.25,-0.25)--(-1,-1);
     \draw[gray, thick] (-1,-1)--(-1.75,-3.25);
      \draw[gray, thick] (-2.25,3.25)--(-1,3);
      \draw[gray, thick] (-1,3)--(-0.75,4.25);
      \draw[gray, thick] (-0.75,4.25)--(1,1);
      \draw [gray, thick,dash dot] (-1,-3) -- (-1,-1);
      \draw [gray, thick,dash dot] (-1,3) -- (-1,1);
       \draw[gray, thick] (-1.75,-3.25)--(-1,-3);
      \draw[gray, thick] (-1,-3)--(-0.75,-3.75);
      \draw[gray, thick] (-0.75,-3.75)--(1,-1);
      \draw[gray, thick] (1,-1)--(3.75,-2.75);
       \draw[gray, thick] (3.75,-2.75)--(3,-1);
       \draw[gray, thick] (3,-1)--(4.75,-0.25);
       \draw[gray, thick] (3,1)--(4.75,-0.25);
       \draw [gray, thick,dash dot] (1,1) -- (3,1);
       \draw [gray, thick,dash dot] (1,-1) -- (3,-1);
       \draw[gray, thick] (3,1)--(4.25,2.75);
       \draw[gray, thick] (4.25,2.75)--(1,1);
       \coordinate[label={$1$}]  (1) at (-1.5,0.75);
       \coordinate[label={$2$}]  (2) at (1.25,1.25);
       \coordinate[label={$3$}]  (2) at (1,-1.75);
       \coordinate[label={$4$}]  (3) at (-1.5,-1.5);
       \coordinate[label={$5$}]  (5) at (-1.25,3.25);
       \coordinate[label={$6$}]  (6) at (3.5,0.75);
       \coordinate[label={$7$}]  (7) at (3.5,-1.5);
       \coordinate[label={$8$}]  (8) at (-1.25,-3.75);
       \coordinate[label={$P_1$}]  (p_1) at (-3.75,-0.75);
       \coordinate[label={$P_2$}]  (p_2) at (-2.75,3);
       \coordinate[label={$P_3$}]  (p_3) at (-0.75,4.15);
       \coordinate[label={$P_4$}]  (p_4) at (4.5,2.75);
        \coordinate[label={$P_5$}]  (p_5) at (5.25,-0.75);
         \coordinate[label={$P_6$}]  (p_6) at (4,-3.5);
          \coordinate[label={$P_7$}]  (p_7) at (-0.75,-4.75);
           \coordinate[label={$P_8$}]  (p_3) at (-2.25,-3.5);
            \draw[red, very thick](0.647,.529)--(-1,1);
         \draw[red, very thick](0.647,0.529)--(-1,3);
         \draw[red, very thick](0.333,-0.333)--(-1,-3);
         \draw[red, very thick](0.333,-0.333)--(-1,-1);
         \draw[red, very thick](0.895,0.649)--(1,1);
          \draw[red, very thick](0.895,0.649)--(3,1);
          \draw[red, very thick](0.805,-0.122)--(3,-1);
          \draw[red, very thick](0.805,-0.122)--(1,-1);
          \draw[red, very thick](0.805,-0.122)--(0.75,-0.0357);
          \draw[red, very thick](0.333,-0.333)--(0.75,-0.0357);
          \draw[red, very thick](0.647,0.529)--(0.75,0.4722);
           \draw[red, very thick](0.75,-0.0357)--(0.75,0.4722);
           \draw[red, very thick](0.895,0.649)--(0.75,0.4722);
\end{tikzpicture}
\end{center}
\end{fig}
\end{minipage}
\bigskip

All segments in the cut locus are portions of perpendicular bisectors $\perp_{\a,\b}$ of the segments joining $P_\a$ and $P_\b$. One needs to consider how various $\perp_{\a,\g}$ intersect to decide when a portion of $\perp_{\a,\b}$ is closer to $P_\a$ and $P_\b$ than to other $P_\g$'s. An example of this is discussed in Section \ref{maplesec}.

\section{Determination of a cut locus}\label{maplesec}

We used {\tt Maple} to help us find the cut locus for many points in  quadrant $Q_1$. We illustrate here with the top half of the cut locus of the point $P=(x,y)=(1.5,0.5)$. Substituting these values into the equations (\ref{eqsvw}), we obtain the coordinates of the points $P_\a=(v_\a,w_\a)$ for $1\le \a\le 8$. The equation of the perpendicular bisector, $\perp_{\a,\b}$, of the segment connecting points $P_\a$ and $P_\b$ is
\begin{equation}\label{perpeq}\begin{cases}w=\dfrac{w_\a+w_\b}2+\dfrac{v_\a-v_\b}{w_\b-w_\a}\bigl(v-\dfrac{v_\a+v_\b}2\bigr)&\{\a,\b\}\ne\{1,5\}\\
v=-x&\{\a,\b\}=\{1,5\}.\end{cases}\end{equation}

{\tt Maple} plots a selected batch of these lines $\perp_{\a,\b}$ in a specified grid. The grid in Figure \ref{maple} is $[-3.5,0.5]\times[0,3]$. Here we have included just those relevant for the top half of the cut locus, which appears in red. Other lines such as $\perp_{2,4}$ and $\perp_{2,5}$ would usually be considered for possible relevance. Trying to do this sort of  analysis for the top and bottom halves of the cut locus together leads to an unwieldy collection of perpendicular bisectors. In Section \ref{endsec}, we show that it suffices to consider the top and bottom parts separately. When crucial intersection points are very close together, we can change the grid to effectively zoom in.

\bigskip
\begin{minipage}{6in}
\begin{fig}\label{maple}

{\bf Finding a cut locus.}

\begin{center}

\begin{\tz}[scale=2.4]
\draw [color=red] [ultra thick] (0,0) -- (0,1.206) -- (1.0127,2.4569) -- (2,2.636);
\draw [color=red] [ultra thick] (1.0127,2.4569) -- (1.184,3);
\draw [color=red] [ultra thick] (0,1.206) -- (-1.6304,2.2608) -- (-3,2.75);
\draw [color=red] [ultra thick] (-1.6304,2.2608) -- (-2,3);
\draw (1.864,0) -- (0,1.206) --  (0,3);
\draw (2,.964) -- (-1.6304,2.2608) -- (-3,3.147);
\draw (-2,1.91) -- (1.0127,2.4569) -- (1.452,3);
\draw (-.976,0) -- (0,1.206);
\draw (-.5,0) -- (-1.6304,2.2608);
\draw (.237,0) -- (1.0127,2.4569);
\node at (-2.04,2.9) {$2$};
\node at (-1.88,2.9) {$3$};
\node at (-3,2.82) {$2$};
\node at (-3,2.68) {$1$};
\node at (-.75,1.77) {$3$};
\node at (-.84,1.62) {$1$};
\node at (1.95, 2.7) {$4$};
\node at (1.95,2.49) {$5$};
\node at (1.1,2.9) {$3$};
\node at (1.22,2.9) {$4$};
\node at (.55,2) {$3$};
\node at (.64,1.91) {$5$};
\node at (-.08,.7) {$1$};
\node at (.08,.7) {$5$};

\end{\tz}
\end{center}
\end{fig}
\end{minipage}

\bigskip
Points equidistant from $P_\a$ and $P_\b$ lie on $\perp_{\a,\b}$. In Figure \ref{maple}, the line $\perp_{\a,\b}$ is annotated with $\a$ on one side and $\b$ on the other, indicating the side closer to $P_\a$ or $P_\b$. The Voronoi cell for a point $P_\a$ is bounded by portions of lines $\perp_{\a,\b}$ for various $\b$, with $\a$ on the cell side of each $\perp_{\a,\b}$. For example, the Voronoi cell for $P_3$ is bounded by portions of $\perp_{3,2}$, $\perp_{3,1}$, $\perp_{3,5}$, and $\perp_{3,4}$, reading from left to right in Figure \ref{maple}.

Although we use the various $\perp_{\a,\b}$ to determine the cut loci, the eventual description of the cut locus is in terms of the corner points of the cube at certain vertices of the cut locus. As seen in Figure \ref{star}, the corner points on lines $\perp_{1,2}$, $\perp_{2,3}$, $\perp_{3,4}$, and $\perp_{4,5}$ are 1, 5, 2, and 6, respectively, and so the top half of the cut locus of the point $P=(1.5,0.5)$ is as depicted in Figure \ref{toph}.

\bigskip
\begin{minipage}{6in}
\begin{fig}\label{toph}

{\bf Top half of a cut locus.}

\begin{center}

\begin{\tz}[scale=.8]
\draw (0,1) -- (0,2) -- (1,3) -- (2.5,3.5);
\draw (1,3) -- (1,4);
\draw (0,2) -- (-1,3) -- (-1,4);
\draw (-1,3) -- (-2.5,3.5);
\node at (-2.7,3.5) {$1$};
\node at (-1,4.2) {$5$};
\node at (1,4.2) {$2$};
\node at (2.7,3.5) {$6$};
\end{\tz}
\end{center}
\end{fig}
\end{minipage}

\bigskip
\section{Proofs}\label{pfsec}
In this section, we show how the regions and curves and their cut loci are obtained.

The coordinate systems are as described in Section \ref{background}.  Let $[8]=\{1,2,3,4,5,6,7,8\}$.
For $P=(x,y)\in Q_1$ and $\a\in [8]$, let $P_\a$ be the point in the star unfolding described earlier. Its coordinates $(v_\a,w_\a)$ are linear expressions, (\ref{eqsvw}), in $x$ and $y$.
In Figure \ref{star}, we depict a typical star unfolding. The vertices $P_1,\ldots,P_8$ are the focal points for the Voronoi cells, while the vertices with numbered labels correspond to the corner points of the cube. 

For $\a,\b\in[8]$, let $\perp_{\a,\b}\!(P)$ denote the perpendicular bisector of the segment connecting $P_\a$ and $P_\b$. Its equation is (\ref{perpeq}).
Note that $\perp_{\a,\a+1}$ has as its extreme point in the star unfolding the corner point 1, 5, 2, 6, 7, 3, 8, and 4, for $\a=1,\ldots,8$. Although our results about $\bL$ are described in terms of the corner points, our work is done in terms of the $\a$.

For $S=\{\a,\b,\g\}\subset[8]$, let $\pi_S(P)$ denote the intersection of $\perp_{\a,\b}\!(P)$ and $\perp_{\b,\g}\!(P)$ (and $\perp_{\a,\g}\!(P)$, as $\pi_S(P)$ is the center of the circle passing through $P_\a$, $P_\b$, and $P_\g$.). Let $L_P$ denote the cut locus of $P$. It is formed from
portions of various $\perp_{\a,\b}(P)$  which are closer to $P_\a$ and $P_\b$ than to any other $P_\g$. The degree-3 vertices of $L_P$ are certain $\pi_S(P)$. From now on, we will usually omit the $(P)$ and the set symbols in subscripts.

A transition from one isomorphism class of $L_P$ to another as $P$ varies will occur when $\pi_{\a,\b,\g}$ passes through another $\perp_{\b,\d}$. This is illustrated in Figure \ref{trans}. The references there to $t$ and $t_0$ will be used in Section \ref{endsec}. In the left side of the figure, $\pi_{\a,\b,\g}$ is part of $L_P$ since it is closer to $P_\a$, $P_\b$, and $P_\g$ than to any other $P$-point, but as $P$ changes and $\pi_{\a,\b,\g}$ moves across $\perp_{\b,\d}$, it is now closer to $P_\d$, and so is not part of $L_P$.

\bigskip
\begin{minipage}{6in}
\begin{fig}\label{trans}

{\bf Transition.}

\begin{center}

\begin{\tz}[scale=.4]
\draw [color=red] [ultra thick] (2,-8) -- (0,-4) -- (0,1) -- (6,4);
\draw [color=red] [ultra thick] (0,1) -- (-6,4);
\draw [color=red] [ultra thick] (0,-4) -- (-2,-8);
\node at (1.8,-4) {$\pi_{\a,\b,\g}$};
\draw (-5,6) -- (0,-4) -- (5,6);
\draw (-6,0) -- (6,0);
\draw (-6,-2) -- (0,1) -- (6,-2);
\draw (0,1) -- (0,6);
\draw (0,-4) -- (0,-8);
\node at (.4,-2) {$\a$};
\node at (-.4,-2) {$\g$};
\node at (-1.9,-7) {$\g$};
\node at (-1.1,-7) {$\b$};
\node at (1.1,-7) {$\b$};
\node at (1.9,-7) {$\a$};
\node at (6,4.4) {$\d$};
\node at (6,3.6) {$\a$};
\node at (5,.4) {$\d$};
\node at (5,-.4) {$\b$};
\node at (-5.5,4) {$\d$};
\node at (-6,3.3) {$\g$};
\node at (0,-9) {$t<t_0$};
\node at (12,-9) {$t=t_0$};
\node at (24,-9) {$t>t_0$};
\draw [color=red] [ultra thick] (9,-6) -- (12,0) -- (15,-6);
\draw [color=red] [ultra thick] (7,2.5) -- (12,0) -- (17,2.5);
\draw (7,0) -- (17,0);
\draw (9,6) -- (12,0) -- (15,6);
\draw (12,-6) -- (12,6);
\draw (7,-2.5) -- (12,0) -- (17,-2.5);
\node at (8.6,-6) {$\g$};
\node at (9.4,-6) {$\b$};
\node at (11.6,-5) {$\g$};
\node at (12.4,-5) {$\a$};
\node at (14.5,-6) {$\b$};
\node at (15.4,-5.8) {$\a$};
\node at (7,2) {$\g$};
\node at (7,2.9) {$\d$};
\node at (17,2.1) {$\a$};
\node at (17,2.9) {$\d$};
\node at (17,-.4) {$\b$};
\node at (17,.4) {$\d$};
\draw [color=red] [ultra thick] (19,-6) -- (22,0) -- (26,0) -- (29,-6);
\draw [color=red] [ultra thick] (19,1.5) -- (22,0) -- (26,0) -- (30,2);
\draw (26,0) -- (30,0);
\draw (19,0) -- (22,0);
\draw (24,-6)  -- (24,6);
\draw (22,0) -- (25,6);
\draw (26,0) -- (23,6);
\draw (19,-3.5) -- (26,0);
\node at (26,4) {$\pi_{\a,\b,\g}$};
\draw (22,0) -- (30,-4);
\node at (23.6,-4) {$\g$};
\node at (24.4,-4) {$\a$};
\node at (24.5,.4) {$\d$};
\node at (24.5,-.4) {$\b$};
\node at (28.5,-6) {$\b$};
\node at (29.4,-5.8) {$\a$};
\node at (18.6,-6) {$\g$};
\node at (19.4,-6) {$\b$};
\node at (28,1.4) {$\d$};
\node at (28,.6) {$\a$};
\node at (19,1.9) {$\d$};
\node at (19,1) {$\g$};
\end{\tz}
\end{center}
\end{fig}
\end{minipage}

\bigskip
\ni We will show in Section \ref{endsec} that this type of transition is the only way to change from one $\bL$ to another.

We  assume first that $\{\a,\b,\g\}$ does not contain both 1 and 5.
Then the $v$-coordinate of $\pi_{\a,\b,\g}$ is found by equating the right hand side of (\ref{perpeq}) for $(\a,\b)$ and  $(\b,\g)$, using the formulas for $v_\a$, $w_\a$, etc.,
in terms of $x$ and $y$ given in (\ref{eqsvw}). This  yields a formula for $v=v_{\a,\b,\g}$ in terms of $x$ and $y$. 

The relationship between $x$ and $y$ such that $\pi_{\a,\b,\g}$ and $\pi_{\a,\b,\d}$ coincide (and hence a transition might occur) is the equation $v_{\a,\b,\g}=v_{\a,\b,\d}$. This yields a fourth-degree equation. We let {\tt Maple}
do the work. For example, we find  the equation for  $(\a,\b,\g,\d)=(2,3,4,5)$ as follows.

\begin{verbatim}v[1]:=-16-x: w[1]:=-y: v[2]:=-12-y: w[2]:=12+x: v[3]:=-8+x: w[3]:=16+y: 
v[4]:=12+y: w[4]:=12-x: v[5]:=16-x: w[5]:=-y:

a:=2: b:=3: c:=4: d:=5:

A:=solve((w[a]+w[b])/2+(v[a]-v[b])/(w[b]-w[a])(v-(v[a]+v[b])/2)
=(w[a]+w[c])/2+(v[a]-v[c])/(w[c]-w[a])(v-(v[a]+v[c])/2),v):

B:=solve((w[a]+w[b])/2+(v[a]-v[b])/(w[b]-w[a])(v-(v[a]+v[b])/2)
=(w[a]+w[d])/2+(v[a]-v[d])/(w[d]-w[a])(v-(v[a]+v[d])/2),v):

simplify(numer(A)denom(B)-numer(B)denom(A))
\end{verbatim}

\medskip

\ni Note that the expressions for $A$ and $B$ will be rational expressions, and so the last line gives a polynomial which equals 0.

For $(a,b,c,d)=(2,3,4,5)$, it yields $$(y^3+(3x+12)y^2+(x^2+40x-16)y+3x^3-44x^2+304x-192)(4+y-x)(=0).$$
\ni The cubic factor  is the second of the five equations listed in (\ref{five}). If we use $(a,b,c,d)=(1,2,3,4)$, we obtain the vertical reflection of the $EF$ curve of (\ref{five}).

For $\{1,\b,\g,5\}$,  since $\perp_{1,5}$ is the line $v=-x$, we do A above for 1, b, and c, omit B, and simplify (numer(A)+x$\cdot$denom(A)). This yields (beginning a practice of often omitting commas)
\begin{align}
1235&\quad x^3-4x^2+(y^2+8y-80)x-4y^2+64(=0)\nonumber\\
1245&\quad x(x^2+y^2+24y+16)(=0)\label{15}\\
1345&\quad x^3-12x^2+(y^2+24y+112)x+4y^2-64(=0).\nonumber\end{align}

For each of these five cases, if 2, 3, and 4 are replaced by 8, 7, and 6, respectively, the equation is obtained by replacing $y$ by $-y$. Altogether we have ten equations. Compare with equations (\ref{five}).

We describe $\bL_P$ for $P$ in the top half of $Q_1$ by the sets $S$ for which $\pi_S$ is a degree-3 vertex of $\bL_P$. Later in this section we will explain how we translate this description to the description involving corner points of the cube, which appeared in Section \ref{stmtsec}. For example, the case $P=(1.5,0.5)$ considered in Section \ref{maplesec} has $\pi_S$ for $S=123$, $135$, and $345$ in its top half. {\tt Maple} plotting of perpendicular bisectors shows that the bottom half of this $\bL_P$ is essentially a flip of the top half, so has $\pi_S$ for $S=178$, $157$, and $567$.

In Section \ref{endsec}, we show that the only possible transitions from one $\bL$ to another are of the type illustrated in Figure \ref{trans}, where an $\a\b\g\d$ intersection bounds one region whose $\bL$ has $\pi_{\a,\b,\g}$ and $\pi_{\a,\g,\d}$ vertices and another with $\pi_{\a,\b,\d}$ and $\pi_{\b,\g,\d}$. The point is that the 4-set\footnote{We use this to denote a set with 4 elements.} defining the bounding curve must have two 3-subsets in each of the regions on either side of it. So, for example, the 1568 curve could not bound a region containing $\bL_{(1.5,0.5)}$ because there are not two of the six 3-sets $S$ for $\bL_{(1.5,0.5)}$ listed in the previous paragraph which are contained in $\{1,5,6,8\}$.

Of the ten equations determined above, all except the ones corresponding to 1568 and 1245 intersect the top half of $Q_1$ in a curve which we denote as $x=\theta_{\a\b\g\d}(y)$ for $0\le y\le4$. Each $y$ has three $x$ values as solutions, but we neglect those that are complex or outside the region $0\le x\le 4-y$. The equation for 1245 does not intersect this region, and the one for 1568 does so only for $0.685\le y\le 1.07$.

{\tt Maple} shows that, for $1.6<y<4$,
$$\theta_{1345}(y)<\theta_{2345}(y)<\theta_{1578}(y)<\theta_{1678}(y)<\theta_{5678}(y)<\theta_{1234}(y)<\theta_{1235}(y)<\theta_{1567}(y),$$
and that for $0\le y\le 4$ all eight of these curves satisfy $0\le\theta_{\a\b\g\d}(y)\le 0.83$.
For $P$ in quadrant $Q_1$,  $\bL_P$ has the type of the case $P=(1.5,0.5)$ considered above, with degree-3 vertices corresponding to $S=123$, $135$, $345$, $178$, $157$, and $567$, until a transition occurs. This will define region $A$ in Figure \ref{figB}.

Now let $1.6<y<4$. Since there are no $\a\b\g\d$ intersections of the eight types in the above string of inequalities in the region $\mathcal R=\{(x,y):0\le y\le4,\,0.83\le x<4-y\}$, and, as noted above, a 1568 intersection cannot affect $\bL_{(1.5,0.5)}$, we conclude that for all $(x,y)$ in $\mathcal R$, $\bL_{(x,y)}=\bL_{(1.5,0.5)}$, with degree-3 vertices 123, 135, 345, 178, 157, and 567. For this, we also need an observation in Section \ref{endsec} that no other $\a\b\g\d$ can have an effect.

As we move from the right, when the point $P=(\theta_{1567}(y),y)$ is encountered, there is a transition from $157$ and $567$ to $156$ and $167$. This is region $G$, with $123$, $135$, $345$, $178$, $156$, and $167$.

Next we encounter $P=(\theta_{1235}(y),y)$, and this causes a transition to $125$, $235$, $345$, $178$, $156$, and $167$. This is region $F$. The next two potential transitions at $1234$ and $5678$ do not effect a change, because neither of these 4-sets contain two 3-subsets which are vertices of region-$F$ cut loci. Next at $P=(\theta_{1678}(y),y)$ we have a transition, changing $178$ and $167$, leading to region $E$ described by $125$, $235$, $345$, $168$, $156$, and $678$. The next potential transition,   $1578$, does not effect a change, but then $2345$ does, to $125$, $234$, $245$, $168$, $156$, and $678$ in region $D$.
Finally, $1345$ does not effect a change because it does not have two 3-sets of region $D$.

Before we discuss other ranges of values of $y$, we point out that when a curve is crossed, it gives a degree-4 vertex of the cut locus, as shown in the middle part of Figure \ref{trans}. Thus, for points $P$ on the curve $GA$ separating regions $G$ and $A$, $\bL_P$ has vertices abutting regions $123$, $135$, $345$, $178$, and $1567$ of the star unfolding, and similarly for points on the other curves crossed in the above analysis. We also note that $\theta_{1567}(1.6)=0.8=\theta_{1235}(1.6)$.

The same procedure is followed for other intervals of values of $y$, arranging the 4-sets $S$ according to the order of $\theta_S(y)$, and then working from right-to-left to see whether the transitions are effective, i.e., whether $S$ contains two 3-sets which are vertices of the region under consideration.
For $0.7085<y<1.6$, the only change from the above order which causes a different transition is that $\theta_{1235}(y)$ is now greater than $\theta_{1567}(y)$, so the 1235 change takes place first, leading to region $H$ with vertices $125$, $235$, $345$, $157$, $567$, and $178$.

The most interesting point is $(6-2\sqrt7,6-2\sqrt7)\approx(.7085,.7085)$, which lies on all of $\theta_{1567}$, $\theta_{1678}$, $\theta_{1568}$,  $\theta_{1578}$, and $\theta_{5678}$.\footnote{To see this remarkable fact, recall that these five curves are obtained by replacing $y$ by $-y$ in the polynomials in (\ref{15}) and the paragraph preceding it. After doing this, let $y=x$. Each of the resulting polynomials equals $x^2-12x+8$ times a linear factor.} These five curves reverse their order at $y=6-2\sqrt7$. For $6-2\sqrt7<y<.715$,
$$\theta_{2345}(y)<\theta_{1578}(y)<\theta_{1678}(y)<\theta_{5678}(y)<\theta_{1567}(y)<\theta_{1568}(y)<\theta_{1234}(y)<\theta_{1235}(y),$$
which has the transitions described in the preceding paragraph, but for $.7045<y<6-2\sqrt7$,
$$\theta_{2345}(y)<\theta_{1568}(y)<\theta_{1567}(y)<\theta_{5678}(y)<\theta_{1678}(y)<\theta_{1578}(y)<\theta_{1234}(y)<\theta_{1235}(y),$$
which has a different order of transitions. Let $.7045<y<6-2\sqrt7$. After the 1235 change, the next one is 1578, leading to region $C$ with vertices $125$, $235$, $345$, $158$, $567$, and $578$. The next transition is due to 5678, leading to region $I$ with vertices $125$, $235$, $345$, $158$, $568$, and $678$. The next transition is due to 1568, which brings us into region $E$, with vertices  $125$, $235$, $345$, $168$, $156$, and $678$, which were already seen when considering larger values of $y$. Finally, a 2345 transition brings us into region $D$ as above.

The 2345 and 1568 curves intersect at $y\approx 0.7045$, so for $y<0.7045$, the 2345 transition precedes the 1568 transition, leading to region $B$ with vertices $125$, $234$, $245$, $158$, $568$, and $678$. 
For $y>.685$, there will be a 1568 transition into region $D$, but for $y<.685$, there is no 1568 transition since $\theta_{1568}(y)<0$ if $y<.685$.

This completes the description of the regions of the top half of quadrant $Q_1$ with constant $\bL$, described in terms of the Voronoi cells. Now we translate this description into one which has the cube's corner numbers at the leaves, which is the description given in Section \ref{stmtsec}, and is needed for giving permuted descriptions in other quadrants. In Figure \ref{top}, we show how the top half of the cut loci appear in terms of Voronoi cells, and list the regions in Figure \ref{figB} in which they appear.
Each edge leading to a leaf is a perpendicular bisector separating Voronoi cells $i$ and $i+1$ for some $i$ mod 8. For $i=1,2,3,4$, the corner point at the end of this bisector is 1, 5, 2, 6, respectively, as can be seen in Figure \ref{star}. The reader can check that this labeled diagram is consistent with  the $\bL$ in Figure \ref{figC}.

\bigskip
\begin{minipage}{6in}
\begin{fig}\label{top}

{\bf Top half of cut loci.}

\begin{center}

\begin{\tz}[scale=.5]
\draw (0,2) -- (0,3) -- (1,4) -- (2,4);
\draw (1,4) -- (1,5);
\draw (0,3) -- (-1,4) -- (-2,4);
\draw (-1,4) -- (-1,5);
\node at (-1,2.5) {$1$};
\node at (1,3) {$5$};
\node at (0,4) {$3$};
\node at (-1.5,4.5) {$2$};
\node at (1.5,4.5) {$4$};  
\node at (0,0) {$A,G$};
\node at (0,6) {$123,\ 135,\ 345$};
\draw (8,1) -- (8,3) -- (9,4) -- (10,4);
\draw (8,2) -- (6.5,2);
\draw (8,3) -- (7,4);
\draw (9,4) -- (9,5);
\node at (7.2,1.5) {$1$};
\node at (7.2,2.6) {$2$};
\node at (8,4) {$3$};
\node at (9.8,4.5) {$4$};
\node at (9.4,2.8) {$5$};
\node at (8,0) {$C,E,F,H,I$};
\node at (8,6) {$125,\ 235,\ 345$};
\draw (16,1) -- (16,4) -- (17,5);
\draw (16,2) -- (14.5,2);
\draw (16,3) -- (17.5,3);
\draw (16,4) -- (15,5);
\node at (15.2,1.5) {$1$};
\node at (15.2,3) {$2$};
\node at (16.7,2) {$5$};
\node at (16,4.9) {$3$};
\node at (16.7,3.8) {$4$};
\node at (16,0) {$B,D$};
\node at (16,6) {$125,\ 234,\ 245$};
\end{\tz}
\end{center}
\end{fig}
\end{minipage}
\bigskip

In Figure \ref{bottom}, we do the same thing for the bottom half of cut loci in the top half of $Q_1$.
The corner numbers at the ends of segments bounding Voronoi cells 5 and 6, 6 and 7, 7 and 8, and 8 and 1 are 7, 3, 8, and 4, respectively.

\bigskip
\begin{minipage}{6in}
\begin{fig}\label{bottom}

{\bf Bottom half of cut loci.}

\begin{center}

\begin{\tz}[scale=.45]
\draw (0,5) -- (0,4) -- (1,3) -- (2,3);
\draw (1,3) -- (1,2);
\draw (0,4) -- (-1,3) -- (-2,3);
\draw (-1,3) -- (-1,2);
\node at (0,2.3) {$7$};
\node at (1.5,2.2) {$6$};
\node at (-1.5,2.2) {$8$};
\node at (-1,4.5) {$1$};
\node at (1,4.5) {$5$};
\node at (0,0) {$A,H$};
\node at (0,6) {$157,\, 567,\, 178$};
\draw (7,5) -- (7,3) -- (8,2) -- (9,2);
\draw (7,4) -- (5.5,4);
\draw (7,3) -- (6,2);
\draw (8,2) -- (8,1);
\node at (6.3,4.7) {$1$};
\node at (8,3.6) {$5$};
\node at (8.7,1.2) {$6$};
\node at (7,1.3) {$7$};
\node at (6.2,3.3) {$8$};
\node at (7,0) {$C$};
\node at (7,6) {$158,\, 567,\, 578$};
\draw (14,5) -- (14,2) -- (15,1);
\draw (14,2) -- (13,1);
\draw (14,3) -- (15.5,3);
\draw (14,4) -- (12.5,4);
\node at (14,1) {$7$};
\node at (14.7,2.2) {$6$};
\node at (13.3,3) {$8$};
\node at (14.7,4) {$5$};
\node at (13.3,4.7) {$1$};
\node at (14,0) {$B,I$};
\node at (14,6) {$158,\, 568,\, 678$};
\draw (20,5) -- (20,2) -- (21,1);
\draw (20,3) -- (18.5,3);
\draw (20,4) -- (21.5,4);
\draw (20,2) -- (19,1);
\node at (20,1) {$7$};
\node at (20.7,3) {$6$};
\node at (20.7,4.7) {$5$};
\node at (19.3,4) {$1$};
\node at (19.3,2) {$8$};
\node at (20,0) {$D,E$};
\node at (20,6) {$156,\, 168,\, 678$};
\draw (27,5) -- (27,3) -- (26,2) -- (25,2);
\draw (27,4) -- (28.5,4);
\draw (27,3) -- (28,2);
\draw (26,2) -- (26,1);
\node at (25.3,1) {$8$};
\node at (27,2) {$7$};
\node at (26.2,3) {$1$};
\node at (27.7,4.7) {$5$};
\node at (27.7,3.2) {$6$};
\node at (27,0) {$F,G$};
\node at (27,6) {$156,\, 167,\, 178$};
\end{\tz}
\end{center}
\end{fig}
\end{minipage}
\bigskip

A similar discussion could be made for the $\bL$ associated to the curves. But it is easier and more insightful to note how the $\bL$ for a curve bounding two regions is obtained from that of each of the two regions by collapsing a segment in which the two regions differ. For example, the $\bL$ for the $BD$ curve in Figure \ref{curveL} is obtained from those in region $B$ or $D$ in Figure \ref{figC} by collapsing the segment connecting the edges leading to corner points 4 and 7. Similarly, the $\bL$ for points of intersection of two curves is obtained by collapsing a segment in the $\bL$ of each. For example, the $\bL$ for point $BDEI$ in Figure \ref{pts} is obtained from those of curves $BD$ and $EI$ in Figure \ref{curveL} by collapsing in each the highest vertical interval.

The $\bL$'s in Figures \ref{two} and \ref{final} are different from those seen previously in that they have a corner point labeling a degree-2 vertex. In these cases, the choice of cuts is not unique, but, of course, the cut locus does not depend on the choice. We comment briefly on the $\bL$ in these cases.

If $P$ is on the left edge of the cube, the $\bL$ is as seen in Figure \ref{Cube}. If $P$ is at a corner point of the cube, the cut locus consists of segments from the corner point opposite $P$ to each of the other corner points. If $P$ is at the center of a face $F$, the cut locus consists of the diagonals of the opposite face $F^\text{op}$ and the four edges of the cube connecting $F$ and $F^\text{op}$.

If $P=(x,4-x)$ with $0<x<4$ is on the half-diagonal, then $\perp_{4,5}$ is the line $w=4$, which intersects the point in the star-unfolding corresponding to corner point 2. Then the short segment connecting the point $\pi_{3,4,5}$ in Figure \ref{cutloc} with the point labeled 2 will have collapsed to a point. In the $A$ diagram in Figure \ref{figC}, this is the collapse of the vertical segment from the point labeled 2. This can be seen in terms of the Voronoi cells in the $A$-part of Figure \ref{top}. A similar thing happens to the vertical segment leading to the point labeled 8, as the equation of $\perp_{7,8}$ is $v=-8$. Also the lines $\perp_{1,7}$, $\perp_{3,5}$, $\perp_{1,5}$, $\perp_{1,3}$, and $\perp_{5,7}$ all intersect at $(v,w)=(-x,4-x)$.

In Section \ref{stmtsec}, we discussed how a permutation $\tau$ (resp.~$\sigma$) applied to corner points yields $\bL$ in the vertical flip (resp.~90-degree clockwise rotation) of a region or curve. Here we give a brief explanation of the reason for that. Such a motion applied to a point $P$ in the region has the same effect on geodesics from $P$, and hence on $L_P$. Referring to Figure \ref{cuts}, we see that, for example, the corner point $8$ in $L_P$ will be replaced by 5 (resp.~7), which expands to the asserted permutations.

\section{No other transitions}\label{endsec}
In this section, we present a proof that there are no regions other than those described earlier. 

Suppose $\bL_{P_0}\ne \bL_{P_1}$. Let $P(t)=(1-t)P_0+tP_1$, and, for any 3-subset $S$ of $[8]$, let $\pi_S(t)=\pi_S(P(t))$, a path in the $vw$-plane. For each $S$ such that $\pi_S(0)$ is a vertex of $L_{P_0}$, let
$$t_0(S)=\sup\{t\in[0,1]: \pi_S(t')\text{ is a vertex of }L_{P(t')}\ \forall t'<t\},$$
and let
$$t_0=\min\{t_0(S):\pi_S(0)\text{ is a vertex of }L_{P_0}\}.$$
Finally, let $S=\{\a,\b,\g\}$ satisfy $t_0(S)=t_0$. The first transition in moving from $P_0$ to $P_1$ will  involve $\pi_S(P(t_0))$.

\begin{prop} There exists $\d\in[8]-S$ and a decomposition of $S$ as $\{\b\}\cup\{\a,\g\}$ such that $\pi_{\{\a,\g,\d\}}(0)$ is a vertex of $L_{P_0}$, and $\pi_S(t_0)=\pi_{\{\a,\g,\d\}}(t_0)$ is a common vertex of $L_{P(t_0)}$.\end{prop}
\begin{proof} There exists $\eps>0$ and $\d\in [8]-S$ such that for $t$ in the interval $(t_0,t_0+\eps)$, $\pi_S(t)$ is closer to $P_\d$ than it is to $P_\a$, $P_\b$, and $P_\g$. The path $\pi_S$ crosses $\perp_{\d,\eta}\!(P(t_0))$ for some $\eta\in[8]$, and $\eta$ must equal $\a$, $\b$, or $\g$, since for $t$ in some interval $(t_0-\eps',t_0)$, $\pi_S(t)$ is closer to $P_\a$, $P_\b$, and $P_\g$ than it is to any other $P_\eta$. Without loss of generality, say $\eta=\b$. Then $\pi_S(t_0)$ intersects $\perp_{\b,\d}\!(P(t_0))$, and so all six perpendicular bisectors from $\{\a,\b,\g,\d\}$ intersect in $L_{P(t_0)}$. By minimality of $t_0$, since $\pi_{\{\a,\g,\d\}}(t_0)\in L_{P(t_0)}$, we conclude that $\pi_{\{ \a,\g,\d\}}(0)\in L_{P_0}$. See Figure \ref{trans} for a depiction of this transition.\end{proof}

\begin{thm} There are no transitions except those claimed earlier in the manuscript.\end{thm}
\begin{proof} Let $S_1=\{2,3,4\}$, $S_2=\{1,5\}$, and $S_3=\{6,7,8\}$. Recall that all of our asserted regions in $Q_1$ have $\bL$ with three vertices from $S_1\cup S_2$ and three from $S_2\cup S_3$. If $\bL_{P_0}\ne \bL_{P_1}$ with $\bL_{P_0}$ in one of our regions, and $t_0$ is as above, so that we are considering the first transition in moving from $P_0$ to $P_1$, then the set $\{\a, \b, \g,\d\}$ involved in the transition must either contain $S_2$ or else equal one of $\{1,2,3,4\}$, $\{2,3,4,5\}$, $\{1,6,7,8\}$, or $\{5,6,7,8\}$. This is true since sets with elements of type $S_1S_2S_3S_3$, $S_1S_1S_2S_3$, $S_1S_1S_3S_3$, $S_1S_1S_1S_3$, or $S_1S_3S_3S_3$ do not contain two 3-subsets of the type of the vertices of $L_{P_0}$. In our earlier determination of the regions in $Q_1$, we considered the four specific sets listed above (containing a single 1 or 5), and also all sets with elements of type $S_1S_1S_2S_2$ and $S_2S_2S_3S_3$. It remains to consider $\{1,5,\a,\b\}$ with $\a\in S_1$ and $\b\in S_3$. If $P\in Q_1$, we use {\tt Maple}, similarly to (\ref{15}), to see that, for $\a\in S_1$ and $\b\in S_3$, $\pi_{\{1,\a,\b\}}(P)$ does not lie on $\perp_{1,5}\!(P)$. Thus there can be no transitions other than the ones described earlier in the paper.

We explain briefly the {\tt Maple} work that led to this conclusion. We follow steps that led to (\ref{15}) but using one of $\{\b,\g\}$ in $S_1$ and one in $S_3$. We obtain equations similar to (\ref{15}). We plot them and find that there are no solutions satisfying $-4<y<4$, $0<x<4-|y|$. 
\end{proof}

\def\line{\rule{.6in}{.6pt}}

\end{document}